\documentclass{amsart}

\usepackage{amssymb, mathtools, enumerate, cancel, tikz, float}
\usepackage[british]{babel}
\usepackage[hidelinks]{hyperref}
\usepackage{cleveref}
\usepackage{microtype}
\usepackage{centernot}
\usetikzlibrary{arrows}

\DeclareMathOperator{\A}{A}
\DeclareMathOperator{\D}{D}

\DeclareMathOperator{\R}{R}
\DeclareMathOperator{\At}{At}
\newcommand{\id}{1}
\DeclareMathOperator{\down}{\downarrow}

\DeclareMathOperator{\dom}{dom}
\DeclareMathOperator{\ran}{range}
\newcommand{\rest}{\mathbin \vartriangleright}
\newcommand{\arest}{\rest'}
\newcommand{\compo}{\mathbin{;}}
\newcommand{\compl}[1]{\overline{#1}}
\newcommand{\intersect}{\wedge}
\newcommand{\bjoin}{\vee}
\newcommand{\meet}{\bigwedge}
\newcommand{\join}{\bigvee}
\newcommand{\stcomp}[1]{{#1}^{\mathsf{c}}}
\newcommand{\defn}[1]{\textbf{#1}}
\newcommand{\algebra}[1]{\mathfrak{#1}}

\newcommand{\powerset}{\raisebox{0.45ex}{$\wp$}}
\newcommand{\pref}{\sqcup}
\newcommand{\update}[2]{#1[#2]}
\newcommand{\updatesymb}{[\phantom x]}

\newcommand\from\colon

\theoremstyle{plain}
\newtheorem{theorem}{Theorem}[section]
\newtheorem{proposition}[theorem]{Proposition}
\newtheorem{lemma}[theorem]{Lemma}
\newtheorem{corollary}[theorem]{Corollary}

\theoremstyle{definition}
\newtheorem{definition}[theorem]{Definition}
\newtheorem{example}[theorem]{Example}
\newtheorem{remark}[theorem]{Remark}
\newtheorem{problem}[theorem]{Problem}

\title[Complete representation for antidomain restriction]{Complete representation by partial functions for signatures containing antidomain restriction}
\author{Brett McLean}
\address{Department of Mathematics: Analysis, Logic and Discrete Mathematics, Ghent University, Ghent, Belgium}
\email{brett.mclean@ugent.be}
\thanks{This work has been suppor\-ted by  SNSF--FWO Lead Agen\-cy Grant 200021L\_196176/\allowbreak G0E2121N and by FWO Postdoctoral Fellowship 1280024N}

\begin{document}

\begin{abstract}
We investigate notions of complete representation by partial functions, where the operations in the signature include antidomain restriction and may include composition, intersection, update, preferential union, domain, antidomain, and set difference. When the signature includes both antidomain restriction and intersection, the join-complete and the meet-complete representations coincide. Otherwise, for the signatures we consider, meet-complete is strictly stronger than join-complete. A necessary condition to be meet-completely representable is that the atoms are separating. For the signatures we consider, this condition is sufficient if and only if composition is not in the signature. For each of the signatures we consider, the class of (meet\mbox{-)}{\allowbreak}completely representable algebras is not axiomatisable by any existential-universal-existential first-order theory. For 14 expressively distinct signatures, we show, by giving an explicit representation, that the (meet\mbox{-)}{\allowbreak}completely representable algebras form a basic elementary class, axiomatisable by a universal-existential-universal first-order sentence. The signatures we axiomatise are those containing antidomain restriction and any of intersection, update, and preferential union and also those containing antidomain restriction, composition, and intersection and any of update, preferential union, domain, and antidomain.

\end{abstract}

\maketitle

\section{Introduction}

The research program on \emph{algebras of partial functions} studies partial functions \emph{abstractly}. To be more precise: (i) the objects of study are collections of partial functions closed under some natural operations such as composition or intersection, allowing us to view these collections as algebraic structures; and (ii) we study properties of these algebraic structures that are \emph{invariant under isomorphism}.

Since functions are one of the most fundamental concepts of the exact sciences, it is not surprising that algebras of partial functions appear in many different areas of mathematics and computer science. In algebra, they arise naturally as inverse semigroups~\cite{wagnergeneralised}, pseudogroups \cite{LAWSON2013117}, and skew lattices~\cite{Leech19967}, and within computing appear in the theory of finite state transducers \cite{10.11452984450.2984453}, computable functions \cite{JACKSON2015259}, deterministic propositional dynamic logics \cite{DBLP:journalsijacJacksonS11}, and separation logic \cite{disjoint}.

The precise set of operations the algebras are equipped with varies between these varied domains of application. However, once the set of operations has been fixed, the most pressing task is to axiomatise the class of all (isomorphs of) algebras of partial functions, if this is indeed possible. Often, these classes have turned out to be finitely axiomatisable varieties or quasivarieties \cite{schein,1018.20057,1182.20058,DBLP:journalsijacJacksonS11}. 

In \cite{10.1093jigpaljzac058}, using a uniform method of representation, Jackson and Stokes provided finite equational or quasiequational axiomatisations of the algebras for around 30 different signatures of operations containing the \emph{domain restriction} operation. Only a handful of these classes of algebras had previously been axiomatised.

Two important stricter conditions we can impose on an isomorphism to an algebra of partial functions, are to require that it be \emph{meet complete} or to require that it be \emph{join complete}. The isomorphism is meet complete if it turns any existing infima into intersections and join complete if it turns any existing suprema into unions. Hence we can define classes of `meet-completely representable'  and of `join-completely representable' algebras. In the field of duality theory, these completely representable algebras have appeared as the appropriate class to use for obtaining \emph{discrete dualities}~\cite{diff-rest2}. Within algebraic logic, completely representable algebras have been studied for other types of representation, such as representation by sets \cite{egrot} or by relations \cite{journalsjsymlHirschH97a,hirsch_2007}. In these contexts, the classes of completely representable algebras are not always first-order axiomatisable.



In this paper we investigate complete representation by partial functions for 22 expressively distinct signatures containing the \emph{antidomain restriction} operation. These signatures form a subset of those considered by Jackson and Stokes in \cite{10.1093jigpaljzac058}.

In \Cref{rep} we give basic definitions. In \Cref{rep2} we show that meet-complete representations are always join-complete and determine for which signatures these two notions coincide.

In \Cref{atomicity} we use the notion of an \emph{atomic representation} to begin characterising the classes of (meet-)completely representable algebras. We show that a necessary condition to be meet-completely representable is that the atoms of the algebra can separate any pair of distinct elements. In particular, the algebra must be atomic, which we use to prove that the classes of (meet-)completely representable algebras are not closed under subalgebras, directed unions or homomorphic images and are not axiomatisable by any existential-universal-existential first-order theory.

In \Cref{dist} we investigate the validity of various distributive laws for the classes of representable, join-completely representable, and meet-completely representable algebras in signatures containing both antidomain restriction and \emph{composition}. This enables us to give examples of algebras of these signatures that are representable and whose atoms are separating, but which are not (meet-)completely representable. Thus the atoms being separating is not a sufficient condition for a representable algebra of these signatures to be (meet-)completely representable.

In \Cref{arep} we present an explicit representation, which we use, in \Cref{axiom}, to prove our main results: for 14 expressively distinct signatures, the class of (meet\mbox{-)}completely representable algebras is a basic elementary class, axiomatisable by a universal-existential-universal first-order sentence (\Cref{theorem:no_composition} and \Cref{theorem:composition}). Only two of these complete representation classes had previously been axiomatised \cite{complete,borlido2022difference}.

In \Cref{section:problems}, we conclude by mentioning some open problems.

\section{Algebras of partial functions and their representations}\label{rep}

In this section we give preliminary definitions. Given an algebra $\algebra{A}$, when we write $a \in \algebra{A}$ or say that $a$ is an element of $\algebra{A}$, we mean that $a$ is an element of the domain of $\algebra{A}$. Similarly for the notation $S \subseteq \algebra{A}$ or saying that $S$ is a subset of $\algebra{A}$. We follow the convention that algebras are always nonempty. If $S$ is a subset of the domain of a map $\theta$ then $\theta[S]$ denotes the set $\{\theta(s) \mid s \in S\}$. If $S_1$ and $S_2$ are subsets of the domain of a binary operation $*$ then $S_1 * S_2$ denotes the set $\{s_1 * s_2 \mid s_1 \in S_1 \text{ and } s_2 \in S_2\}$. In a poset $\algebra P$ (whose identity should be clear) the notation $\down a$ signifies the down set $\{b \in \algebra P \mid b \leq a\}$.

We begin by making precise what is meant by partial functions and
algebras of partial functions. 
\begin{definition}
  Let $X$ and $Y$ be sets. A \defn{partial function} from $X$ to $Y$
  is a subset $f$ of $X \times Y$ validating
\begin{equation*}
  (x, y) \in f \text{ and } (x, z) \in f \implies y = z.
\end{equation*}
If $X = Y$ then $f$ is called simply a partial function on $X$.  Given
a partial function $f$ from $X$ to $Y$, its \defn{domain} is the set
\[\dom(f) \coloneqq \{x \in X \mid \exists \ y \in Y \colon (x, y) \in f\}.\]
The \defn{range} of $f$ is the set \[\ran(f) \coloneqq \{y \in Y \mid \exists \ x \in X \colon (x, y) \in f\}.\]
\end{definition}

\begin{definition}\label{first}
Let $\sigma$ be an algebraic signature whose symbols are a subset of $\{\arest, \compo, \intersect, \updatesymb, \pref, \D, \A\}$. An \defn{algebra of partial functions} of the signature $\sigma$ is an algebra of the signature $\sigma$ whose elements are partial functions and with operations given by the set-theoretic operations on those partial functions described in the following.

Let $X$ be the union of the domains and ranges of all the partial functions. We call $X$ the \defn{base}. The operations are defined as follows.
\begin{itemize}
\item
The binary operation $\arest$ is \defn{antidomain restriction}. It is the restriction of the second argument to elements \emph{not} in the domain of the first; that is:
\[ f \arest g \coloneqq \{(x, y) \in X^2 \mid x \not\in \dom(f)
  \text{ and } (x, y) \in g\}\text{.}\]
\item
The binary operation $\compo$ is \defn{composition} of partial functions:
\[f \compo g = \{(x,z) \in X^2 \mid \exists y \in X  (x, y) \in f\text{ and }(y, z) \in g\}\text{.}\]
\item
The binary operation $\wedge$ is \defn{intersection}:
\[f \intersect g = \{(x,y) \in X^2 \mid (x, y) \in f\text{ and }(x, y) \in g\}\text{.}\]
\item
The binary operation $\updatesymb$ is \defn{update}:
\[ \update f g(x) =
\begin{cases}
f(x) & \text{if }f(x)\text{ defined but }g(x)\text{ undefined}\\
g(x) & \text{if }f(x)\text{ defined and }g(x)\text{ defined}\\
\text{undefined} & \text{otherwise.}
\end{cases}\]
\item
The binary operation $\pref$ is \defn{preferential union}\footnote{This operation is also known as \emph{override}.}:
\[(f \pref g)(x) =
\begin{cases}
f(x) & \text{if }f(x)\text{ defined}\\
g(x) & \text{if }f(x)\text{ undefined, but }g(x)\text{ defined}\\
\text{undefined} & \text{otherwise.}
\end{cases}\]
\item
The unary \defn{domain} operation $\D$ is the operation of taking the diagonal of the domain of a function:
\[\D(f) = \{(x, x) \in X^2 \mid x \in \dom(f)\}\text{.}\]
\item
The unary \defn{antidomain} operation $\A$ is the operation of taking the diagonal of the antidomain of a function\textemdash those elements of $X$ where the function is not defined:
\[\A(f) = \{(x, x) \in X^2 \mid x \in X \setminus \dom(f)\}\text{.}\]
\end{itemize}
\end{definition}

In this paper we only consider signatures that contain antidomain restriction. Note that with antidomain restriction present, two further commonly considered operations on partial functions are term definable. 
\begin{itemize}
\item
The constant $0$ is the nowhere-defined \defn{zero} function:
\[0 = \varnothing\text{.}\]
We can define $0$ with the term $f \arest f$.
\item
The binary operation $\rest$ is \defn{domain restriction}.\begin{NoHyper}\footnote{This operation is also known as \emph{restrictive multiplication}.}\end{NoHyper} It is the restriction of the second argument to the domain of the first; that is:
\[ f \rest g \coloneqq \{(x, y) \in X^2 \mid x \in \dom(f)
  \text{ and } (x, y) \in g\}\text{.}\]
  We can define $f \rest g$ with the term $(f \arest g) \arest g$.
\end{itemize}
Additionally, given that antidomain restriction is in the signature, the presence of the following are equivalent, respectively, to the presence of antidomain and the presence of intersection; hence we do not need to consider them independently.
\begin{itemize}
\item
The constant $\id$ is the \defn{identity} function on $X$:
\[\id = \{(x, x) \in X^2 \}\text{.}\]
We have $\id = \A(0)$ and $\A(f) = f \arest \id$.
\item
The binary operation $\setminus$ is \defn{relative complement}:
\[f \setminus g = \{(x,y) \in X^2 \mid (x, y) \in f\text{ and }(x, y) \not\in g\}\text{.}\]
We have $f \setminus g = (f \intersect g) \arest f$ and $f \intersect g = f \setminus (f \setminus g)$.
\end{itemize}

The list of operations we have given does not exhaust those that have been considered for partial functions, but does include many of the most commonly appearing operations. Notable exceptions are \emph{range} and related operations, such as \emph{range restriction} and \emph{antirange}.

\begin{definition}
Let $\algebra{A}$ be an algebra of one of the signatures specified by \Cref{first}. A \defn{representation of $\algebra{A}$ by partial functions} is an isomorphism from $\algebra{A}$ to an algebra of partial functions of the same signature. If $\algebra{A}$ has a representation then we say it is \defn{representable}.
\end{definition}

The following theorem is stated for precisely the signatures we investigate in this paper.

\begin{theorem}[Jackson and Stokes \cite{10.1093jigpaljzac058}]\label{thm:jackson-stokes}
Let $\{\arest\} \subseteq \sigma \subseteq \{\arest, \intersect, \updatesymb, \pref\}$ or $\{\arest, \compo\} \subseteq \sigma \subseteq \{\arest, \compo, \intersect, \updatesymb, \pref, \D, \A\}$. Then the class of $\sigma$-algebras representable by partial functions is a finitely based variety or a finitely based quasivariety.
\end{theorem}

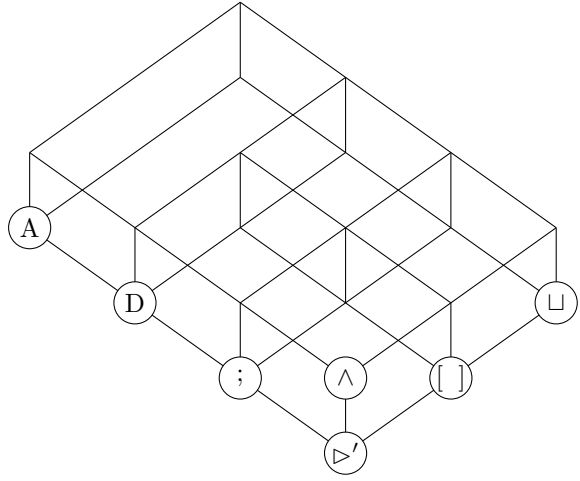
\begin{figure}[!htbp]
\centering
\begin{tikzpicture}[xscale=1.4]
\draw (0,3)--(3,0);
\draw (0,4)--(3,1);
\draw (2,3)--(4,1);
\draw (2,4)--(4,2);
\draw (2,5)--(5,2);
\draw (2,6)--(5,3);
\foreach \i in {0,1,2,3}
{
\draw (3-\i,0+\i)--(5-\i,2+\i);
\draw (3-\i,1+\i)--(5-\i,3+\i);

\draw (3-\i,0+\i)--(3-\i,1+\i);
\draw (5-\i,2+\i)--(5-\i,3+\i);
}
\draw (4,1)--(4,2);
\draw (3,2)--(3,3);
\draw (2,3)--(2,4);
\filldraw[color=black, fill=white](3,1) ellipse (.2 and .28);
\node  at (3,1) {$\intersect$};
\filldraw[color=black, fill=white](3,0) ellipse (.2 and .28);
\node at (3,0) {$\arest$};
\filldraw[color=black, fill=white](4,1) ellipse (.2 and .28);
\node  at (4,1) {$\updatesymb$};
\filldraw[color=black, fill=white](5,2) ellipse (.2 and .28);
\node  at (5,2) {$\pref$};
\filldraw[color=black, fill=white](2,1) ellipse (.2 and .28);
\node  at (2,1) {$\compo$};
\filldraw[color=black, fill=white](1,2) ellipse (.2 and .28);
\node  at (1,2) {$\D$};
\filldraw[color=black, fill=white](0,3) ellipse (.2 and .28);
\node  at (0,3) {$\A$};
\end{tikzpicture}
\caption{Hasse diagram of the signatures investigated and their relative expressiveness. Each vertex represents the signature containing the operations appearing below the vertex.}\label{figure:signatures}
\end{figure}

The signatures covered by \Cref{thm:jackson-stokes} are depicted in \Cref{figure:signatures}. The diagram uses the same format as \cite[Figure~1]{10.1093jigpaljzac058} (which depicts signatures $\{\rest\} \subseteq \sigma \subseteq \{\rest,\arest, \intersect, \updatesymb, \pref\}$).

 Note that $\D$ is definable from $\A$, as $\D = \A^2$, and thus $\D$ and $\A$ do not require independent dimensions in the diagram. Similarly, $\updatesymb$ is term definable in signatures containing $\arest$ and $\pref$, as $\update f g = f \rest (g \pref f)$ (using the definition of $\rest$ in terms of $\arest$); thus $\updatesymb$ and $\pref$ do not require independent dimensions. For the signatures containing $\compo$ and $\A$ (in which $\arest$ is definable), the converse is true: we can give a term definition of $\pref$ if we have $\updatesymb$. A definition is $f \pref g =\update{\update {\A(\A(f) \compo \A(g))} g} f$. Thus there are only four signatures containing $\compo$ and $\A$ to consider: $\{\compo, \A\}$, $\{\compo, \A, \intersect\}$, $\{\compo, \A, \pref\}$, and $\{\compo, \A, \pref, \intersect\}$. For these four signatures, finite equational/quasiequational axiomatisations were first presented in \cite{DBLP:journalsijacJacksonS11}. Of the other signatures, $\{\arest, \pref\}$ had previously been axiomatised in \cite{leech1990skew}, $\{\arest,\intersect, \pref\}$ (in the guise of $\{\setminus, \pref\}$) had previously been axiomatised in \cite{cirulis2011nearlattices}, and $\{\arest, \intersect\}$ (in the guise of $\{\rest, \setminus\}$) had previously been axiomatised \cite{borlido2022difference}.
 
We now note that no further relations between expressiveness of signatures hold than shown in \Cref{figure:signatures}. That all the vertices for signatures including $\{\arest, \compo, \intersect, \updatesymb, \pref\}$ are distinct is shown in \cite{10.1093jigpaljzac058}. Essentially the same arguments can be used to show that the six signatures containing $\D$ but not $\A$ are distinct from one another. It is easy to verify that none of the signatures including $\{\arest, \compo, \intersect, \updatesymb, \pref\}$ can express $\D$, and hence these two groups are disjoint from one another. That the four signatures containing $\A$ are distinct from one another is noted in \cite{DBLP:journalsijacJacksonS11}. It only remains to argue that this group of four is distinct from the remainder of the signatures. For that it suffices to exhibit a collection of partial functions closed under $\arest$, $\compo$, $\intersect$, $\updatesymb$, $\pref$, and $\D$ but not closed under $\A$. Choose any infinite set $X$ as the base. Then the set of all identity functions whose domains are finite subsets of $X$ is such a collection.

\section{Complete representations}\label{rep2}

In this section, we give the definitions of the central notions in this paper: join-complete representations and meet-complete representations. We then begin to analyse these notions by determining the relationship between them.

 If an algebra of a signature containing $\arest$ is representable by partial functions, then it forms a poset when equipped with the relation $\leq$ defined by 
\begin{equation}\label{define_order}a \leq b \iff a \rest b = a.\end{equation}
This relation corresponds to $\subseteq$ in the sense that for \emph{any} representation $\theta$, we have \begin{equation}\label{represent-poset}a \leq b \iff \theta(a) \subseteq \theta(b).\end{equation}
When we treat an algebra representable by partial functions as a poset, we always mean the poset with partial order defined by equation \eqref{define_order}. Note that the constant $0$ (defined by the term $a \arest a$) is always the least element with respect to this order.

The next two definitions apply to any situation where the concept of a representation has been defined and such that the representable algebras can be viewed as posets.\footnote{However, the definitions are only interesting when the partial order and notion of representation are defined in such a way that equation \eqref{represent-poset} is necessarily valid. Hence we may take equation \eqref{represent-poset} as the \emph{definition} of a representation of a poset.} So in particular, these definitions apply to representations as fields of sets as well as to representations by partial functions.

\begin{definition}\label{def:join}
A representation $\theta$ of a poset $\algebra{P}$ over the base $X$ is \defn{join complete} if, for every subset $S$ of $\algebra{P}$, if $\join S$ exists, then
\[\theta(\join S) = \bigcup \theta[S].\]
\end{definition}

\begin{definition}\label{def:meet}
A representation $\theta$ of a poset $\algebra{P}$ over the base $X$ is \defn{meet complete} if, for every nonempty subset $S$ of $\algebra{P}$, if $\meet S$ exists, then
\[\theta(\meet S) = \bigcap \theta[S]\text{.}\]
\end{definition}

Note that $S$ is required to be nonempty in \Cref{def:meet}, but not in \Cref{def:join}. For representations of Boolean algebras as fields of sets, the notions of meet complete and join complete are equivalent, so in this case we may simply use the adjective \defn{complete}.


The following lemma demonstrates the utility of signatures $\sigma$ containing $\arest$. The similarity of representable $\sigma$-algebras to Boolean algebras allows results from the theory of Boolean algebras to be imported into the setting of $\sigma$-algebras.

\begin{lemma}\label{lemma:boolean}
Let $\sigma$ be a signature containing $\arest$ and let $\algebra{A}$ be a $\sigma$-algebra. If $\algebra{A}$ is representable by partial functions, then for every $a \in \algebra{A}$, the set $\down a$, with least element $0$, greatest element $a$, meet given by $\rest$ and complementation given by $\compl{b} \coloneqq b \arest a$ is a Boolean algebra. Any representation $\theta$ of $\algebra{A}$ by partial functions restricts to a representation of $\down a$ as a field of sets over $\theta(a)$. The representation of $\down a$ is complete if either
\begin{enumerate}
\item\label{item:meet}
$\theta$ is a meet-complete representation;
\item\label{item:join}
$\sigma$ contains $\intersect$, and $\theta$ is a join-complete representation.
\end{enumerate}
\end{lemma}

\begin{proof}
If $\theta$ is a representation of $\algebra{A}$ by partial functions, then $b \leq a \iff \theta(b) \subseteq \theta(a)$, so $\theta$ does indeed map elements of $\down a$ to subsets of $\theta(a)$.

 Since $\algebra A$ isomorphic to an algebra of partial functions, we can reason using properties of partial functions to deduce that $b, c \in \down a \implies b \rest c \in \down a$. Then to see that $\rest$ is represented as intersection on $\down a$, first note that $\theta(b \rest c) = \theta(b \rest (c \rest a)) = \{(x,y) \in \theta(a) \mid x \in \dom(b) \cap \dom(c)\}$, by the definition of functional representability. Then since both $\theta(b)$ and $\theta(c)$ are subset of $\theta(a)$, we have $\{(x,y) \in \theta(a) \mid x \in \dom(b) \cap \dom(c)\}= \theta(b) \cap \theta(c)$.

 For $b \leq a$
\[\theta(\compl{b}) = \theta(b \arest a) = \theta(b) \arest \theta(a) =\{(x,y) \in \theta(a) \mid x \in \dom(b)\} = \theta (a) \setminus \theta (b),\]
since $\theta(b) \subseteq \theta(a)$. So $\compl{b} \in \down a$ (as $\theta$ is an isomorphism) and $\theta(\compl{b}) = \stcomp{\theta(b)}$, where the set complement is taken relative to $\theta(a)$. 

We have verified that the restriction of $\theta$ to $\down a$ is a representation of $(\down a, 0, a, \rest, \compl{\phantom{c}})$ as a field of sets over $\theta(a)$. It follows that $(\down a, 0, a, \rest, \compl{\phantom{c}})$ is a Boolean algebra.

Suppose $\theta$ is meet complete. If $S$ is a nonempty subset of $\down a$, then all lower bounds for $S$ in $\algebra{A}$ are also in $\down a$. Hence if $\meet_{\down a} S$ exists then it equals $\meet_{\algebra{A}} S$, and so $\theta(\meet_{\down a} S) = \bigcap \theta[S]$. So the representation of $\down a$ is complete.

Suppose instead that $\sigma$ contains $\intersect$, and $\theta$ is join complete. Suppose $S \subseteq \down a$ and $\join_{\down a} S$ exists. If $c \in \algebra{A}$ and $c$ is an upper bound for $S$, then $c \geq c \intersect a \geq \join_{\down a} S$. Hence $\join_{\down a} S = \join_{\algebra{A}} S$, giving $\theta(\join_{\down a} S) = \theta(\join_{\algebra{A}} S) = \bigcup \theta[S]$. So the representation of $\down a$ is complete.
\end{proof}

Let us give a name to the completeness property identified in \Cref{lemma:boolean}.

\begin{definition}
Let $\sigma$ be a signature containing $\arest$, let $\algebra{A}$ be a $\sigma$-algebra, and let $\theta$ be a representation of $\algebra A$ by partial functions. Then $\theta$ is \defn{locally complete} if for all $a \in \algebra A$, the map $\theta$ restricts to a complete representation of the Boolean algebra $\down a$.
\end{definition}

\begin{lemma}\label{lemma:local_join}
Let $\sigma$ be a signature containing $\arest$. Let $\algebra{A}$ be a $\sigma$-algebra and $\theta$ be a representation of $\algebra{A}$ by partial functions. If $\theta$ is locally complete, then it is join complete.
\end{lemma}

\begin{proof}
Suppose that $\theta$ is locally complete. Let $S$ be a subset of $\algebra{A}$ and suppose that $\join_{\algebra{A}} S$ exists. Let $a = \join_{\algebra{A}} S$. Then $\join_{\algebra{A}} S \in \down a$, so $\join_{\algebra{A}} S = \join_{\down a} S$. Then
\[\theta(\join_{\algebra{A}} S) = \theta(\join_{\down a} S) = \bigcup \theta[S]\text{.}\qedhere\]
\end{proof}

Thus we have \emph{meet complete} $\implies$ \emph{locally complete} $\implies$ \emph{join complete}. When $\intersect$ is also in the signature, then all three notions are equal.



\begin{corollary}[of \Cref{lemma:boolean}\eqref{item:join}]
Let $\sigma$ be a signature including $\{\arest, \intersect\}$. Let $\algebra{A}$ be a $\sigma$-algebra and $\theta$ be a representation of $\algebra{A}$ by partial functions. If $\theta$ is join complete, then it is meet complete.
\end{corollary}

\begin{proof}
Suppose that $\theta$ is join complete. Let $S$ be a nonempty subset of $\algebra{A}$ and suppose that $\meet_{\algebra{A}} S$ exists. As $S$ is nonempty, we can find $s \in S$. Then
\[\theta(\meet_{\algebra{A}} S) = \theta(\meet_{\algebra{A}} (S \intersect \{s\})) = \theta(\meet_{\down{s}} (S \intersect \{s\})) = \bigcap \theta[S \intersect \{s\}] = \bigcap \theta[S]\text{.}\qedhere\]
\end{proof}

These results tell us that, just as for representations of Boolean algebras, when a signature $\sigma$ includes $\{\arest, \intersect\}$, we can describe representations of $\sigma$-algebras by partial functions as \defn{complete}, without any risk of confusion about whether we mean meet complete or join complete.\footnote{Since $\{\arest, \intersect\}$ is equivalent to $\{\rest, {\setminus}\}$, this observation is not new, having already been made in \cite{borlido2022difference}.}

For our signatures that omit $\intersect$ (that is, for the lower layer in \Cref{figure:signatures}), the three notions of completeness of a representation are distinct.  The following example shows that for these signatures, join-complete representations do not necessarily restrict to complete representations of the down-set Boolean algebras. 

\begin{example}[join complete $\centernot\implies$ locally complete]\label{example1}
Consider the following concrete algebra of partial functions, $\algebra{F}$. (We will clarify the signature shortly.) The base for $\algebra F$ is the disjoint union of a two-element set, $\{\infty_1, \infty_2\}$, and $\mathbb{N}$.   We write $\mathrm{id}_A$ for the identity function on a set $A$. The elements of $\algebra{F}$ are precisely the partial functions of either of the forms
\begin{itemize}
\item $\mathrm{id}_A$, where $A$ is a finite subset of $\mathbb{N}$;

\item $\mathrm{id}_A \cup f$, where $A$ is a cofinite subset of $\mathbb{N}$ and $f$ is a permutation on $\{\infty_1, \infty_2\}$.
\end{itemize}

One can check that $\algebra{F}$ is closed under the operations of composition, preferential union, and antidomain. Thus we can view $\algebra F$ as a $\sigma$-algebra of partial functions for any signature $\sigma$ whose operations are definable by a term in the signature $\{\compo, \pref, \A\}$, in particular for any $\sigma \subseteq \{\arest, \compo,  \updatesymb, \pref, \D, \A\}$. It is easy to verify that the identity function $\iota \from \algebra F \to \algebra F$ is a join-complete representation, as follows. Let $S \subseteq \algebra F$. We must check that either $S$ has no join in $\algebra F$, or it has a join equal to $\bigcup S$.
\begin{itemize}
\item
If $S$ contains \emph{both} an element extending the identity on $\{\infty_1, \infty_2\}$ and an element extending the other permutation on $\{\infty_1, \infty_2\}$, then $S$ has no upper bounds, so in particular no least upper bounds.
\item
If $S$ contains precisely \emph{one} of these types of elements, then $\bigcup S \in \algebra F$, and this is necessarily the join of $S$.
\item
If $S$ contains \emph{neither} of these types of elements, then if $\bigcup S$ is finite, it belongs to $\algebra F$, so again is necessarily the join of $S$. Otherwise, if $\bigcup S$ is infinite then $S$ has two distinct minimal upper bounds formed from the union of $\bigcup S$ with the two possible permutations on $\{\infty_1, \infty_2\}$; thus $S$ has no join.
\end{itemize}
 However, the representation $\iota \from \algebra F \to \algebra F$ does not restrict to a complete representation of the Boolean algebra $\down (\mathrm{id}_{\mathbb N} \cup \mathrm{id}_{\{\infty_1, \infty_2\}})$, since $\meet \{ \mathrm{id}_{\{i, i+1,\dots\}} \cup \mathrm{id}_{\{\infty_1, \infty_2\}} \mid i \in \mathbb N\}= \emptyset$, but $\bigcap \{ \mathrm{id}_{\{i, i+1,\dots\}} \cup \mathrm{id}_{\{\infty_1, \infty_2\}} \mid i \in \mathbb N\}= \mathrm{id}_{\{\infty_1, \infty_2\}}$. Thus $\iota \from \algebra F \to \algebra F$ is not locally complete.
\end{example}

That meet complete is a strictly stronger notion than locally complete can be demonstrated with an algebra with only three elements.

\begin{example}[locally complete $\centernot\implies$ meet complete]\label{example:three-element}
On a two-element base set, take the identity function, one of the constant functions, and the empty function. This collection is closed under $\compo$, $\pref$, and $\A$, and the identity representation $\iota$ is locally complete, but does not represent the meet $0$ of the two nonzero elements as their intersection (which is nonempty), but rather as $\emptyset$.
\end{example}

The main results of the paper will be axiomatisations of the classes of algebras that are (meet-)completely representable by partial functions for signatures $\{\arest\} \subseteq \sigma \subseteq \{\arest, \intersect, \updatesymb, \pref\}$ (\Cref{theorem:no_composition}) or $\{\arest, \compo, \intersect\} \subseteq \sigma \subseteq \{\arest, \compo, \intersect, \updatesymb, \pref, \D, \A\}$ (\Cref{theorem:composition}); thus for 14 expressively distinct signatures (of which for 11 the various notions of complete representation coincide). Two of these complete representation classes have previously been axiomatised: for the signature $\{\compo, \intersect, \A\}$ in \cite{complete}, and for the signature $\{\arest, \intersect\}$ (in the guise of $\{\rest, \setminus\}$) in \cite{borlido2022difference}.

\section{Atomicity}\label{atomicity}

We begin our investigation of the complete representation classes by considering properties related to atoms, both for algebras and for representations.

\begin{definition}
Let $\algebra{P}$ be a poset with a least element, $0$. An \defn{atom} of $\algebra{P}$ is a minimal nonzero element of $\algebra{P}$. We write $\At(\algebra{P})$ to denote the set of atoms of $\algebra{P}$. We say that $\algebra{P}$ is \defn{atomic} if every nonzero element is greater than or equal to an atom. We say that the \defn{atoms are separating} if whenever $a \not\leq b \in \algebra P$ then there exists $c \in \At(\algebra P)$ with $c \leq a$ and $c \not\leq b$.
\end{definition}


Let $\sigma$ be a signature containing $\arest$. We noted in the proof of \Cref{lemma:boolean} that representations of $\sigma$-algebras necessarily represent the partial order by set inclusion. The following definition is meaningful for any notion of representation where this is the case.

\begin{definition}
Let $\algebra{P}$ be a poset with a least element and let $\theta$ be a representation of $\algebra{P}$. Then $\theta$ is \defn{atomic} if $x \in \theta(a)$ for some $a \in \algebra{P}$ implies $x \in \theta(b)$ for some $b \in \At(\algebra P)$ with $b \leq a$.\footnote{Often less strict forms of this definition are given, in situations where the difference is of no consequence. If intersection is in the signature, then it is not necessary to assume $b \leq a$. If $\algebra P$ has a maximum element $1$, we may assume $a = 1$.} 
\end{definition}

We will need the following theorem.

\begin{theorem}[Hirsch and Hodkinson \cite{journalsjsymlHirschH97a}]\label{thm:hirsch-hodkinson} Let $\algebra{B}$ be a Boolean algebra. A representation of $\algebra{B}$ as a field of sets is atomic if and only if it is complete.
\end{theorem}

Note that being completely representable does not imply a Boolean algebra is complete, but having an atomic representation \emph{does} imply a Boolean algebra is atomic. Hence the existence of Boolean algebras that are atomic but not complete, for example, the finite--cofinite algebra on any infinite set.

\begin{proposition}\label{prop:at-com}
Let $\sigma$ be a signature containing $\arest$. Let $\algebra{A}$ be a $\sigma$-algebra and $\theta$ be a representation of $\algebra{A}$ by partial functions. 
Then $\theta$ is atomic if and only if it is locally complete.
\end{proposition}

\begin{proof}
 Suppose that $\theta$ is locally complete. Let $(x, y)$ be a pair contained in $\theta(a)$ for some $a \in \algebra{A}$. 
 By \Cref{thm:hirsch-hodkinson}, $(x, y) \in \theta(b)$ for some atom $b$ of the Boolean algebra $\down a$. Since an atom of $\down a$ is clearly an atom of $\algebra{A}$, and $b \leq a$, the representation $\theta$ is atomic.
 
Conversely, suppose that $\theta$ is atomic, $a \in \algebra A$, and $S$ is a nonempty subset of $\down a$ such that $\meet S$ exists. It is always true that $\theta(\meet S) \subseteq \bigcap \theta[S]$, regardless of whether or not $\theta$ is atomic. For the reverse inclusion, we have
\[\begin{array}{cll}
 & (x, y) \in \bigcap \theta[S]
\\ \implies & (x, y) \in \theta(s) & \text{for all }s \in S
\\ \implies & (x, y) \in \theta(b) & \text{for some atom }b\text{ such that }(\forall s \in S)\text{ } b \leq s
\\ \implies & (x, y) \in \theta(b) & \text{for some atom }b\text{ such that } b \leq \meet S
\\ \implies &  (x, y) \in \theta(\meet S)\text{.}
\end{array}\]
To see that the third line follows from the second, first take an atom $b$ with $(x, y) \in \theta(b)$---which exists by the second line, since $S \neq \emptyset$. Then we have $(x, y) \in \theta(s \rest b)$ for any $s \in S$. So for all $s \in S$, the element $s \rest b$ is nonzero, so equals $b$, since $b$ is an atom. Since $b, s \in \down a$, it is also the case that $s \rest b \leq s$; hence $b \leq s$.
\end{proof}


\begin{corollary}\label{cor:separating}
Let $\sigma$ be a signature containing $\arest$, and let $\algebra{A}$ be a $\sigma$-algebra. If $\algebra{A}$ is meet-completely representable by partial functions then the atoms of $\algebra A$ are separating.
\end{corollary}

\begin{proof}
Let $a \not\leq b \in \algebra A$. Let $\theta$ be any meet-complete representation of $\algebra{A}$. Then $\theta(a) \not\subseteq \theta(b)$, so there exists $(x, y) \in \theta(a) \setminus \theta (b)$. By \Cref{lemma:boolean}\eqref{item:meet}, the representation $\theta$ is locally complete. Then by \Cref{prop:at-com}, the representation $\theta$ is atomic, so $(x, y) \in \theta(c)$ for some $c \in \At(\algebra{A})$ with $c \leq a$. Since $(x, y) \in \theta(c)$ and $(x,y) \not\in \theta(b)$, we know $\theta(c) \not\subseteq \theta(b)$, and thus $c \not\leq b$.
\end{proof}

\begin{corollary}\label{cor:atomic}
Let $\sigma$ be a signature containing $\arest$, and let $\algebra{A}$ be a $\sigma$-algebra. If $\algebra{A}$ is meet-completely representable by partial functions then $\algebra{A}$ is atomic.
\end{corollary}

\begin{proof}
The atoms being separating is a stronger condition than a poset being atomic.
\end{proof}

So far we have exploited the Boolean algebras that are contained in any representable $\sigma$-algebra, when $\sigma$ contains $\arest$. But we can also travel in the opposite direction and interpret any Boolean algebra as an algebra of any signature $\sigma \subseteq \{\arest, \compo, \intersect, \updatesymb, \pref, \D, \A\}$, by using the Boolean operations to give interpretations to the symbols in $\sigma$ as follows.
\begin{itemize}
\item
$a \arest b \coloneqq \compl a \intersect b$

\item
$a \compo b \coloneqq a \intersect b$
\item
$a \intersect b \coloneqq a \intersect b$
\item
$\update a b \coloneqq a$
\item
$a \pref b \coloneqq a \bjoin b$
\item
$\D(a) \coloneqq a$
\item
$\A(a) \coloneqq \compl a$
\end{itemize} Again this enables us to easily prove results about $\sigma$-algebras using results about Boolean algebras.

We know by the following argument that a Boolean algebra, $\algebra{B}$, viewed as an algebra of a signature $\sigma \subseteq \{\arest, \compo, \intersect, \updatesymb, \pref, \D, \A\}$, is representable by partial functions. By Stone's representation theorem we may assume that $\algebra{B}$ is a field of sets. Then the set of all identity functions on elements of $\algebra{B}$ forms a representation of $\algebra{B}$ by partial functions. Using the same argument, if a Boolean algebra is completely representable as a field of sets then the $\sigma$-algebra obtained from it is completely representable by partial functions. 

Hirsch and Hodkinson used \Cref{thm:hirsch-hodkinson} to identify those Boolean algebras that are completely representable as a field of sets as precisely the atomic Boolean algebras.\footnote{This result, that a Boolean algebra is completely representable if and only if it is an atomic algebra, had also been discovered previously by Abian \cite{Abian01051971}.} Hence a Boolean algebra is completely representable by partial functions if and only if it is atomic. The following proposition uses this fact to prove various negative results about the axiomatisability of classes of completely representable $\sigma$-algebras.

\begin{proposition}\label{closure}
Let $\{\arest\} \subseteq \sigma \subseteq \{\arest, \compo, \intersect, \updatesymb, \pref, \D, \A\}$. The class of $\sigma$-algebras that are (meet-)completely representable by partial functions is not closed with respect to the operations shown in the following table and so is not axiomatisable by first-order theories of the indicated corresponding form.

\hfill \\
\begin{tabular}{r  l  l }

& Operation & Axiomatisation \\[1mm]

(i) & subalgebra & universal \\

(ii) & directed union & universal-existential \\

(iii) & homomorphism & positive \\

\end{tabular}
\end{proposition}

\begin{proof}
The proof is identical to the proof for the particular signature $\{\compo, \intersect, \A\}$ found in \cite{complete}; we reproduce it here for convenience.

 In each case we use the fact, which we noted previously, that a Boolean algebra is completely representable by partial functions if and only if it is atomic.
\begin{enumerate}[(i)]
\item
We show that the class is not closed under subalgebras. It follows that the class cannot be axiomatised by any universal first-order theory. Let $\algebra{B}$ be any non-atomic Boolean algebra, for example the countable atomless Boolean algebra, which is unique up to isomorphism. By Stone's representation theorem we may assume that $\algebra{B}$ is a field of sets, with base $X$ say. Then $\algebra{B}$ is a subalgebra of $\powerset(X)$ and $\powerset(X)$ is atomic, but $\algebra{B}$ is not.
\item
We show that the class is not closed under directed unions. It follows that the class cannot be axiomatised by any universal-existential first-order theory. Again, let $\algebra{B}$ be any non-atomic Boolean algebra. Then $\algebra{B}$ is the union of its finitely generated subalgebras, which form a directed set of algebras. The finitely generated subalgebras, being Boolean algebras, are finite and hence atomic. So we have, as required, a directed set of atomic Boolean algebras whose union is not atomic.
\item
We show that the class is not closed under homomorphic images. It follows that the class cannot be axiomatised by any positive first-order theory. Let $X$ be any infinite set and $I$ the ideal of $\powerset(X)$ consisting of finite subsets of $X$. Then $\powerset(X)$ is atomic, but the quotient $\powerset(X) / I$ is atomless and nontrivial and so is not atomic.
\qedhere
\end{enumerate}
\end{proof}

Since we have mentioned the subalgebra and homomorphism operations, we note that for each signature $\{\arest\} \subseteq \sigma \subseteq \{\arest, \compo, \intersect, \updatesymb, \pref, \D, \A\}$, the class of (meet\mbox{-)}com\-pletely representable $\sigma$-algebras \emph{is} closed under direct products. Indeed, it is routine to verify that given complete representations of each factor in a product we can form a complete representation of the product using disjoint unions in the obvious way.

\begin{proposition}\label{axioms}
Let $\{\arest\} \subseteq \sigma \subseteq \{\arest, \compo, \intersect, \updatesymb, \pref, \D, \A\}$. The class of $\sigma$-algebras that are (meet-)completely representable by partial functions is not axiomatisable by any existential-universal-existential first-order theory. 
\end{proposition}

\begin{proof}
Again, an identical copy of this proof first appeared in \cite{complete}, for the particular signature $\{\compo, \intersect, \A\}$.

Let $\algebra{B}$ be any atomic Boolean algebra with an infinite number of atoms and $\algebra{B'}$ be any Boolean algebra that is not atomic, but that also has an infinite number of atoms. We will show that $\algebra{B'}$ satisfies any existential-universal-existential sentence satisfied by $\algebra{B}$. Since $\algebra{B}$ is completely representable by partial functions and $\algebra{B'}$ is not, this shows that the complete representation class cannot be axiomatised by any existential-universal-existential theory.

We will show that for certain Ehrenfeucht--Fra\"{i}ss\'{e} games, duplicator has a winning strategy. For an overview of Ehrenfeucht--Fra\"{i}ss\'{e} games see, for example, \cite{HodgesW:modt}. Briefly, two players, spoiler and duplicator, take turns to choose elements from two algebras. Duplicator wins if the two sequences of choices determine an isomorphism between the subalgebras generated by all the elements chosen.

Consider the game in which spoiler must in the first round choose $n_1$ elements of $\algebra{B}$, in the second round $n_2$ elements of $\algebra{B'}$ and in the third and final round $n_3$ elements of $\algebra{B}$. Each round, duplicator responds with corresponding choices from the other algebra. Let $\varphi$ be any sentence in prenex normal form whose quantifiers are, starting from the outermost, $n_1$ universals, then $n_2$ existentials and finally $n_3$ universals. It is not hard to convince oneself that if duplicator has a winning strategy for the game then $\algebra{B'} \models \varphi \implies \algebra{B} \models \varphi$. Hence if duplicator has a winning strategy for all games of this form---where spoiler chooses finite numbers of elements from $\algebra{B}$ then $\algebra{B'}$ then $\algebra{B}$---then all universal-existential-universal sentences satisfied by $\algebra{B'}$ are satisfied by $\algebra{B}$. Equivalently, $\algebra{B'}$ satisfies any existential-universal-existential sentence satisfied by $\algebra{B}$, which is what we are aiming to show.

Since our algebras are Boolean algebras, a choice of a finite number of elements from one of the algebras generates a finite subalgebra, with a finite number of atoms. The atoms form a partition, that is, a sequence $(a_1,\ldots,a_n)$ of nonzero elements with $\join_i a_i = 1$ and $a_i \intersect a_j = 0$ for all $i \neq j$. As the game progresses and more elements are chosen, the partition is refined---the elements of the partition are (finitely) further subdivided. The elements the two players have actually chosen are all uniquely expressible as a join of some subset of the partition.

Suppose that, throughout the game, duplicator is able to maintain a correspondence between the partitions on the two algebras. That is, if spoiler subdivides an element $a$ of the existing partition into $(a_1,\ldots,a_n)$ then the element corresponding to $a$ should be partitioned into a corresponding $(a_1',\ldots,a_n')$. Then clearly this determines a winning sequence of moves for duplicator: each of spoiler's choices is the join of some subset of one partition and duplicator's choice should be the join of the corresponding elements of the other partition. At the end of the game there will exist an isomorphism between the generated subalgebras that sends each element chosen during the game to the corresponding choice from the other algebra. Hence a strategy for maintaining a correspondence between the two partitions provides a winning strategy for duplicator.

For an element $a$ of $\algebra{B}$ or $\algebra{B'}$ we will say that $a$ is of size $n$, for finite $n$, if $a$ is the join of $n$ distinct atoms, otherwise $a$ is of infinite size. Duplicator can maintain a correspondence by playing as follows.

\begin{description}
\item[Round 1] (Spoiler plays on atomic algebra, duplicator on non-atomic) Duplicator should simply provide a partition with matching sizes.

\item[Round 2] (Spoiler non-atomic, duplicator atomic) For subdivisions of elements of finite size, duplicator can provide a subdivision with matching sizes. For subdivisions of elements of infinite size, there is necessarily at least one element in the subdivision of infinite size---duplicator should select one such, match everything else with distinct single atoms and match this infinite size element with what remains on the atomic side.

\item[Round 3] (Spoiler atomic, duplicator non-atomic) At the start of this round every element of the partition of the atomic algebra is matched with something of greater or equal size on the non-atomic side. Hence duplicator can easily provide matching subdivisions.\qedhere
\end{description}
\end{proof}

\section{Distributivity}\label{dist}

We now turn our attention to the validity of various distributive laws with respect to the classes of representable and (meet-)completely representable $\sigma$-algebras for signatures $\sigma$ including $\{\arest, \compo\}$. 

\begin{definition}
Let $\algebra{P}$ be a poset and $*$ be a binary operation on $\algebra{P}$. We say that $*$ is \defn{completely right-distributive over joins} if, for any subset $S$ of $\algebra{P}$ and any $a \in \algebra{P}$, if $\join S$ exists, then
\[\join S * a = \join(S * \{a\})\text{.}\]
\end{definition}

\begin{proposition}
Let $\sigma$ be a signature including $\{\arest, \compo\}$, and let $\algebra{A}$ be a $\sigma$-algebra that is representable by partial functions. Then composition is completely right-distributive over joins.
\end{proposition}

\begin{proof}
As $\algebra{A}$ is representable, we may assume the elements of $\algebra{A}$ are partial functions. Let $S$ be a subset of $\algebra{A}$ such that $\join S$ exists and let $a \in \algebra{A}$.

Firstly, for all $s \in S$ we have $\join S \compo a \geq s \compo a$  and so $\join S \compo a$ is an upper bound for $S \compo \{a\}$.

Now suppose that $b \in \algebra{A}$ is an upper bound for $S \compo \{a\}$, that is, for all $s \in S$, we have $b \geq s \compo a$. For $s \in S$, suppose $s$ is defined on $x$ and let $s(x) = y$. If $a$ is defined on $y$, then $s \compo a$ is defined on $x$, so, since $b \geq s \compo a$, in this case $b$ is defined on $x$. If $a$ is \emph{not} defined on $y$ then, as $(\join S)(x) = y$, in this case $\join S \compo a$ is not defined on $x$. Hence the domain of $b \arest (\join S \compo a)$ is disjoint from the domain of $s$. Therefore
\[(b \arest (\join S \compo a)) \arest \join S \geq s\text{.}\]
Since $s$ was an arbitrary element of $S$, we have
\[(b \arest (\join S \compo a)) \arest \join S \geq \join S,\]
and so $b \arest (\join S \compo a)$ is disjoint from the domain of $\join S$. Since the domain of $\join S \compo a$ is included in the domain of $\join S$, we see that $b \arest (\join S \compo a)$ is disjoint from the domain of $\join S \compo a$.
%
Thus
\[(b \arest (\join S \compo a)) \arest (\join S \compo a) = \join S \compo a\text{.}\]
 The left-hand side is $b \rest (\join S \compo a)$, by the definition of $\rest$. Thus $\join S \compo a \leq b$, by the definition of $\leq$. 
As $b$ was an arbitrary upper bound for $S \compo \{a\}$, we conclude that $\join S \compo a$ is the least upper bound for $S \compo \{a\}$.
\end{proof}

\begin{remark}\label{dist-laws}
For $\{\arest, \compo\}$-algebras representable by partial functions it is easy to see that for finite $S$, if $\join S$ exists, then
\begin{align*}
a \compo \join S = \join(\{a\} \compo S)\text{.} && \text{(composition is left-distributive over joins)}\end{align*}
When $\intersect$ is in the signature, the corresponding law for meets also holds. That is, for $\{\arest, \compo, \intersect\}$-algebras representable by partial functions, for finite, \emph{nonempty} $S$,
\begin{align*}a \compo \meet S = \meet(\{a\} \compo S)\text{.} && \text{(composition is left-distributive over meets)}\end{align*}
When $\intersect$ is not in the signature, not even the binary-meet version of this law is valid, as can easily be shown with the algebra from \Cref{example:three-element}.
\end{remark}

We now give an example that shows that these distributive laws cannot, in general, be extended to arbitrary joins and meets. We will use this example to show that for any signature $\{\arest, \compo\} \subseteq \sigma \subseteq \{\arest, \compo, \intersect, \updatesymb, \pref, \D, \A\}$, there exist $\sigma$-algebras that are representable as partial functions, and whose atoms are separating, but that have no atomic representation.

\begin{example}\label{eg:noncdist}
Consider the following concrete algebra of partial functions, $\algebra{F}$. (We will clarify the signature shortly.) The base for $\algebra F$ is the disjoint union of a one-element set, $\{p\}$, and $\mathbb{N}_\infty \coloneqq \mathbb{N} \cup \{\infty\}$. Let $\mathcal{S}$ be all the subsets of $\mathbb{N}_\infty$ that are either finite and do not contain $\infty$ or cofinite and contain $\infty$. Let $f$ be the partial function defined only on $p$ and mapping $p$ to $\infty$.  As before, we write $\mathrm{id}_A$ for the identity function on a set $A$. The elements of $\algebra{F}$ are precisely the partial functions of the form $\mathrm{id}_A \cup g$, where both
\begin{itemize}
\item $A \in \mathcal S$;

\item $g$ is equal to $\varnothing$, $\mathrm{id}_{\{p\}}$, or $f$.
\end{itemize}

One can check that $\algebra{F}$ is closed under the operations of composition, intersection, preferential union, and antidomain. Thus we can view $\algebra F$ as a $\sigma$-algebra of partial functions for any signature $\sigma$ whose operations are definable by a term in the signature $\{\compo, \intersect, \pref, \A\}$. Let us assume that $\{\arest, \compo\} \subseteq \sigma$. One can now also check that the atoms of $\algebra{F}$ are separating.

For $i \in \mathbb{N}$, let $g_i$ be the restriction of the identity to $\{1,\ldots,i\}$. Then $\join_i g_i$ exists and is equal to the identity restricted to $\mathbb{N}_\infty$. So
\[f \compo \join_{i \in \mathbb{N}} g_i = f \neq \varnothing = \join_{i \in \mathbb{N}} (f \compo g_i)\text{.}\]

For $i \in \mathbb{N}$, let $h_i$ be the restriction of the identity to $\{i,\ldots\} \cup \{\infty\}$. Then $\meet_i h_i$ exists and is equal to the nowhere-defined function. So
\[f \compo \meet_{i \in \mathbb{N}} h_i = \varnothing \neq f = \meet_{i \in \mathbb{N}} (f \compo h_i)\text{.}\]
\end{example}

\begin{lemma}\label{lemma:compdist}
Let $\sigma$ be a signature that includes $\{\arest, \compo\}$, and let $\algebra{A}$ be a $\sigma$-algebra. If $\algebra A$ is join-completely representable by partial functions, then composition in $\algebra{A}$ is completely left-distributive over joins. If $\algebra A$ is meet-completely representable by partial functions, then composition in $\algebra{A}$ is completely left-distributive over meets.
\end{lemma}

\begin{proof}
First let $\theta$ be any join-complete representation of $\algebra{A}$. We prove that composition is completely left-distributive over joins. Let $S$ be a subset of $\algebra{A}$ such that $\join S$ exists and let $a \in \algebra{A}$. Suppose that for all $s \in S$ the element $b \in \algebra{A}$ satisfies $b \geq a \compo s$. Then for all $s \in S$ we have $\theta(b) \supseteq \theta(a \compo s)$. Hence
\begin{align*}
\theta(b) &\supseteq \bigcup \theta[\{a\} \compo S] \\
&= \bigcup (\{\theta(a)\} \compo \theta[S]) \\
&= \theta(a) \compo \bigcup \theta[S] \\
&= \theta(a \compo \join S)\text{.} 
\end{align*}
The second equality is a true property of any collection of functions, indeed of any collection of relations. We conclude that $b \geq a \compo \join S$, and hence $a \compo \join S$ is the least upper bound for $\{a\} \compo S$.

Now let $\theta$ be any meet-complete representation of $\algebra{A}$. We prove that composition is completely left-distributive over meets. Let $S$ be a nonempty subset of $\algebra{A}$ such that $\meet S$ exists and let $a \in \algebra{A}$. Suppose that for all $s \in S$, the element $b \in \algebra{A}$ satisfies $b \leq a \compo s$. Then for all $s \in S$, we have $\theta(b) \subseteq \theta(a \compo s)$. Hence
\begin{align*}
\theta(b) &\subseteq \bigcap \theta[\{a\} \compo S] \\
&= \bigcap (\{\theta(a)\} \compo \theta[S]) \\
&= \theta(a) \compo \bigcap \theta[S] \\
&= \theta(a \compo \meet S)\text{.} 
\end{align*}
This time the second equality holds only because we are working with functions. It is not, in general, a true property of relations. We conclude from the above that $b \leq a \compo \meet S$, and hence $a \compo \meet S$ is the greatest lower bound for $\{a\} \compo S$.
\end{proof}

\begin{proposition}
For every signature $\{\arest, \compo\} \subseteq \sigma \subseteq \{\arest, \compo, \intersect, \updatesymb, \pref, \D, \A\}$, there exist $\sigma$-algebras that are representable by partial functions and for which the atoms are separating but that have no atomic representation.
\end{proposition}

\begin{proof}
Let $\algebra{F}$ be the algebra of \Cref{eg:noncdist}, viewed as a $\sigma$-algebra. Since $\algebra{F}$ \emph{is} an algebra of partial functions, it is certainly representable by partial functions. We have already mentioned that the atoms of $\algebra{F}$ are separating. We have demonstrated that composition in $\algebra{F}$ is not completely left-distributive over joins. Hence, by \Cref{lemma:compdist}, $\algebra{F}$ has no join-complete representation. So, by \Cref{lemma:local_join}, $\algebra{F}$ has no locally complete representation. Then by \Cref{prop:at-com}, $\algebra{F}$ has no atomic representation.
\end{proof}

To make the discussion of distributive laws for composition comprehensive we finish by mentioning the one remaining case: right-distributivity of composition over meets. Here the weakest possible results, that the finite version of the law is valid for meet-completely representable algebras, does not hold for representation by partial functions.
It is not necessarily the case that for finite, nonempty $S$,
\begin{align*}
\meet S \compo a = \meet(S \compo \{a\})\text{.} && \text{(composition is right-distributive over meets)}\end{align*}
 In the algebra of partial functions shown in \Cref{picture}, where sub-identity elements are omitted, we have
\[(f_1 \intersect f_2) \compo g = 0 \compo g = 0 \neq h = h \intersect h = (f_1 \compo g) \intersect (f_2 \compo g)\text{.}\]
The algebra is completely representable because it is already an algebra of partial functions (of, in particular, the signature $\{\arest,\compo,\intersect\}$) and it is finite.

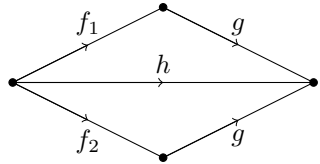
\begin{figure}[H]
\centering
\begin{tikzpicture}
\begin{scope}[->]
\draw (0,0)--(1,0.5);
\draw (0,0)--(1,-0.5);
\draw (0,0)--(2,0);
\draw (2,1)--(3,0.5);
\draw (2,-1)--(3,-0.5);
\end{scope}
\draw (0,0)--(2,1);
\draw (0,0)--(2,-1);
\draw (0,0)--(4,0);
\draw (2,1)--(4,0);
\draw (2,-1)--(4,0);
\draw[fill] (0,0) circle [radius=0.05];
\draw[fill] (4,0) circle [radius=0.05];
\draw[fill] (2,1) circle [radius=0.05];
\draw[fill] (2,-1) circle [radius=0.05];
\node [above] at (1,0.5) {$f_1$};
\node [above] at (3,0.5) {$g$};
\node [below] at (1,-0.5) {$f_2$};
\node [below] at (3,-0.5) {$g$};
\node [above] at (2,0) {$h$};
\end{tikzpicture}
\caption{An algebra refuting right-distributivity over meets}\label{picture}
\end{figure}

We summarise the validity of distributivity laws for composition in \Cref{table:distributivity} and \Cref{table:distributivity2}. In the tables, a cardinal $\kappa$ indicates that $\kappa$ is the least (nonzero) cardinality of a subset for which the law can fail. An $\infty$ indicates that the law holds for arbitrary cardinality subsets.

\begin{table}
\begin{center}
\begin{tabular}{ |l|c|c|c| } 
 \hline
  & representable & join-completely  & meet-completely  \\ 
  \hline
 right-distributive over joins & $\infty$ & $\infty$ & $\infty$\\ 
 left-distributive over joins & $\aleph_0$ & $\infty$ & $\infty$ \\ 
 right-distributive over meets & 2 & 2 & 2 \\ 
 left-distributive over meets & 2 & 2 & $\infty$\\ 
 \hline
\end{tabular}
\vspace{.2cm}
\caption{Validity of distributive laws for composition for $\{\arest, \compo\}$-algebras}\label{table:distributivity}
\end{center}
\end{table}

\begin{table}
\begin{center}
\begin{tabular}{ |l|c|c| } 
 \hline
  & representable &  completely representable \\ 
  \hline
 right-distributive over joins & $\infty$ &  $\infty$\\ 
 left-distributive over joins & $\aleph_0$ &  $\infty$ \\ 
 right-distributive over meets & 2 &  2 \\ 
 left-distributive over meets & $\aleph_0$ &  $\infty$\\ 
 \hline
\end{tabular}
\vspace{.2cm}
\caption{Validity of distributive laws for composition for $\{\arest, \compo,\intersect\}$-algebras}\label{table:distributivity2}
\end{center}
\end{table}

We have not yet mentioned distributive laws for operations other than composition. There is one other operation whose left-distributivity we will need to know to prove the representation theorem of the following section: domain restriction. For left-distributivity over \emph{joins}, the atoms being separating is sufficient to give us the unrestricted cardinality law.

\begin{lemma}\label{lemma:left_dist_rest}
Let $\sigma$ be a signature containing $\arest$, and let $\algebra{A}$ be a $\sigma$-algebra that is representable by partial functions and whose atoms are separating. Then $\rest$ is completely left-distributive over joins.
\end{lemma}

\begin{proof}
Let $S$ be a subset of $\algebra{A}$ such that $\join S$ exists, and let $a \in \algebra{A}$. It is clear that for each $s \in S$ we have $a \rest s \leq a \rest \join S$. Hence $a \rest \join S$ is an upper bound for $\{a\} \rest S$. Now let $b$ be an arbitrary upper bound for $\{a\} \rest S$. We want to show that $a \rest \join S \leq b$. Suppose not; then $a \rest \join S \not\leq (a \rest \join S) \rest b$. Then as the atoms are separating, there exists an atom $x$ with $x \leq a \rest \join S$ and $x \not\leq (a \rest \join S) \rest b$. We now claim that for each $s \in S$, we have $x \not\leq s$. Suppose otherwise; then we have $a \rest x \leq a \rest s \leq b$. Hence $(a \rest \join S) \rest (a \rest x) \leq (a \rest \join S) \rest b$. Since $x \leq a \rest \join S$, we also know, by properties of partial functions, that $(a \rest \join S) \rest (a \rest x)  = x$. Hence $x \leq (a \rest \join S) \rest b$, contradicting the second hypothesis about $x$. This proves the claim. We also know, from $x \leq a \rest \join S$, that $x \leq \join S$. Now in the Boolean algebra $\down \join S$, we have the atom $x$ with $x \leq \join S$ and $\forall s \in S: x \not\leq s$, which is a contradiction. We conclude that $a \rest \join S \leq b$, and thus $a \rest \join S$ is the least upper bound for $\{a\} \rest S$. 
\end{proof}

The hypothesis that the atoms are separating cannot be removed from \Cref{lemma:left_dist_rest}. We outline an example demonstrating this and leave it as an exercise to the reader to check the details. Take the algebra from \Cref{example1} and remove all partial functions that extend $\mathrm{id}_{\{1,\infty_1, \infty_2\}}$. The remaining functions are closed under antidomain restriction, so form a representable $\{\arest\}$-algebra. Let $S$ be $\{\mathrm{id}_A \mid A\text{ is a finite subset of }\mathbb{N}\}$ and $a$ be any element of the algebra whose domain is an infinite set not containing $1$. Then $\join S$ exists, but $\join (\{a\} \rest S)$ does not (as its two upper bounds are incomparable).

 For left-distributivity over \emph{meets}, we only need that the algebra be representable.

\begin{lemma}\label{lemma:left_dist_rest_meets}
Let $\sigma$ be a signature containing $\arest$, and let $\algebra{A}$ be a $\sigma$-algebra that is representable by partial functions. Then $\rest$ is completely left-distributive over meets.
\end{lemma}

\begin{proof}
Let $S$ be a nonempty subset of $\algebra{A}$ such that $\meet S$ exists, and let $a \in \algebra{A}$. It is clear that for each $s \in S$ we have $a \rest \meet S \leq a \rest s$. Hence $a \rest \meet S$ is a lower bound for $\{a\} \rest S$. Now let $b$ be an arbitrary lower bound for $\{a\} \rest S$. Then for each $s \in S$ we have $b \leq a \rest s \leq s$. Hence $b \leq \meet S$. Then clearly $a \rest b \leq a \rest \meet S$. But since $S$ is nonempty, we have $b \leq a \rest s$ for some $s$, from which it is clear that $a \rest b = b$. We conclude that $b \leq a \rest \meet S$, and hence $a \rest \meet S$ is the greatest lower bound for $\{a\} \rest S$.
\end{proof}

\section{A representation}\label{arep}

We have seen that for an algebra of a signature $\sigma$ that includes $\{\arest, \compo\}$ to be \emph{meet}-completely representable by partial functions it is necessary for it to be representable by partial functions and atomic and for composition to be completely left-distributive over meets. In this section, we show that when $\{\arest, \compo\} \subseteq \sigma \subseteq \{\arest, \compo, \intersect, \updatesymb, \pref, \D, \A\}$, these conditions are also sufficient. The representations used for the proof are almost independent of the signature: we use one of two similar representations depending on whether or not $\D$ is expressible. These are Cayley-style representations but also have a certain similarity to the Birkhoff--Stone representation, for the representations use atoms for their base, and atoms correspond to principal ultrafilters in Boolean algebras.

We first need to define the relation corresponding to `have the same domain'. For an algebra $\algebra A$ of a signature containing $\arest$, define
\[a \sim b \iff (a \rest b = b\text{ and }b \rest a = a).\]
If $\algebra A$ is representable by partial functions, then $\sim$ is an equivalence relation corresponding to being represented by functions with the same domain.

\begin{proposition}\label{prop:rep}
Let $\{\arest, \compo\} \subseteq \sigma \subseteq \{\arest, \compo, \intersect, \updatesymb, \pref, \D, \A\}$, and let $\algebra{A}$ be a $\sigma$-algebra. Suppose $\algebra{A}$ is representable by partial functions and the atoms of $\algebra A$ are separating. 
 If $\D, \A \not\in \sigma$, for each $a \in \algebra{A}$, let $\theta(a)$ be the following partial function on the disjoint union $\At(\algebra{A}) \amalg \At(\algebra A) / {\sim}$. For $x \in \At(\algebra A)$
\[\theta(a)(x) =
\begin{cases}
x \compo a & \mathrm{if }\text{ }x \compo a \neq 0 \\
\mathrm{undefined} & \mathrm{otherwise}
\end{cases}\]
and 
\[\theta(a)(x/{\sim}) =
\begin{cases}
x \rest a & \mathrm{if }\text{ }x \rest a \neq 0 \\
\mathrm{undefined} & \mathrm{otherwise}
\end{cases}\]
If $\D \in \sigma$ or $\A \in \sigma$,  for each $a \in \algebra{A}$, let $\theta$ be only the first component of the partial function just defined, so $\theta(a)$ is a partial function on  $\At(\algebra{A})$.
Then $\theta$ is a representation of $\algebra{A}$ by partial functions, with base either $\At(\algebra{A}) \amalg \At(\algebra A) / {\sim}$ or $\At(\algebra{A})$ as appropriate.

Further,
\begin{enumerate}
\item\label{claim1}
 if, in $\algebra A$, composition is completely left-distributive over joins, then $\theta$ is join complete;
 \item\label{claim2}
 if, in $\algebra A$, composition is completely left-distributive over meets, then $\theta$ is meet complete. 
 \end{enumerate}
\end{proposition}

\begin{proof}
Since $\algebra A$ is representable, we may assume it is an algebra of partial functions. 

\smallskip

\noindent\underline{{Well defined}} \,We first need to show that, for each $a \in \algebra{A}$, the partial function $\theta(a)$ is well defined and maps into $\At(\algebra{A}) \amalg \At(\algebra A) / {\sim}$  or $\At(\algebra{A})$ (in fact it always maps into $\At(\algebra{A})$). 

First let $x$ be an atom and $b \in \algebra{A}$, and suppose $b \leq x \compo a$. We must show that $b =0$ or $b= x \compo a$.  As $b \leq x \compo a$, we know $\dom(b) \subseteq \dom(x \compo a) \subseteq \dom(x)$. Since $x$ is an atom, $b \rest x = 0$ or $b \rest x=x$. Knowing that $\dom(b) \subseteq \dom(x)$, we deduce that $b = 0$ or $\dom(b)=\dom(x)$, respectively. In the first case, we are done; in the second, we deduce that $\dom(b) = \dom(x \compo a)$, which, together with $b \leq x \compo a$, yields $b = x \compo a$.

Now let $x$ and $y$ be atoms with $x \sim y$. Then $\dom(x) = \dom(y)$ so it is clear that $x \rest a = y \rest a$. Hence $\theta$ is well defined. We must also show that $x \rest a$ is an atom, that is, given $b \leq x \rest a$, either $b = 0$ or $b = x \rest a$. As before, $\dom(b) \subseteq \dom(x)$ and either $b \rest x = 0$ or $b \rest x = x$, yielding $b = 0 $ or $\dom(b)  = \dom(x)$, respectively. The second case, together with $b \leq x \rest a \leq a$, gives $b = x \rest a$.

\smallskip

\noindent\underline{Injective} \,To show that $\theta$ is injective, let $a$ and $b$ be distinct elements of $\algebra A$. Then without loss of generality, $a \nleq b$. Take an atom $x$ with $x \leq a$ but $x \nleq b$, which exists because the atoms of $\algebra A$ are separating. There are two cases to consider, depending on the signature. If $\D, \A \not \in \sigma$, then by the definition of $\leq$, we have $x \rest a = x \neq x \rest b$, and thus the partial functions $\theta(a)$ and $\theta(b)$ must differ. If $\D \in \sigma$ or $\A \in \sigma$, then from $x \neq 0$ it follows that $\D(x) \neq 0$. In fact, $\D(x)$ must also be an atom, for if $\alpha \leq \D(x)$, then $\alpha \compo x \leq x$, and thus $\alpha \compo x = 0$ or $\alpha \compo x = x$. Since $\alpha$ is a restriction of $\D(x)$ these possibilities give $\alpha = 0$ or $\alpha = \D(x)$ respectively. We now calculate $\theta(a)$ and $\theta(b)$ on $\D(x)$. Since $x \leq a$, we have $\D(x) \compo a = x \rest a = x$, so $\theta(a)(\D(x)) = x$.  Since $x \leq b$, we have $\D(x) \compo b = x \rest b \neq x$, so $\theta(b)(\D(x)) \neq x$, if $\theta(b)(\D(x))$ is defined at all.

\smallskip

\noindent\underline{$\arest$} \,To show that $\theta$ represents $\arest$ correctly, let $a, b \in \algebra{A}$ and $x \in \At(\algebra{A})$. We first check that $\theta(a \arest b)$ and $\theta(a) \arest \theta(b)$ have the same domain. If $\theta(a \arest b)(x)$ is defined, then $x \compo (a \arest b) \neq 0$, so $x \compo b \neq 0$, and thus $\theta(b)(x)$ is defined.  By reasoning about partial functions, we see that $x \compo (a \arest b) \neq 0$ implies $(x \compo a) \arest x \neq 0$, which as $x$ is an atom implies $(x \compo a) \arest x = x$. Again reasoning about partial functions, this implies $x \compo a = 0$, that is, $\theta(a)(x)$ is undefined. Conversely, if $\theta(b)(x)$ is defined and $\theta(a)(x)$ not, that is $x \compo b \neq 0$ and $x \compo a = 0$, then by reasoning about partial functions, $x \compo (a \arest b) \neq 0$, so $\theta(a \arest b)(x)$ is defined. Verifying that $\theta(a \arest b)(x/{\sim})$ is defined if and only if $\theta(a) \arest \theta(b)(x/{\sim})$ is defined is similar. Now to see that the values of $\theta(a \arest b)$ and $\theta(a) \arest \theta(b)$ never differ, note first that if $x \compo (a \arest b)$ and $x \compo b$ are nonzero, then they are equal since  $x \compo b$ is an atom and $x \compo (a \arest b) \leq x \compo b$. Similarly for $x \rest (a \arest b)$ and $x \rest b$.

\smallskip

\noindent\underline{$\,\compo\,$} \,To show that $\theta$ represents $\compo$ correctly, let $a, b \in \algebra{A}$ and $x \in \At(\algebra{A})$. Then clearly $\theta(a \compo b)(x) = \theta(a) \compo \theta(b)(x)$ if both sides are defined. The left-hand side is defined precisely when $x \compo a \compo b$ is nonzero and the right-hand side when $x \compo a$ and $x \compo a \compo b$ are both nonzero. Since $x \compo a \compo b \neq 0$ implies $x \compo a \neq 0$, these conditions coincide. We also need to check the agreement of $\theta(a \compo b)(x/{\sim})$ and $\theta(a) \compo \theta(b)(x/{\sim})$. When both are defined, they are $x \rest(a\compo b)$ and $(x \rest a) \compo b$ respectively; these are clearly equal in any algebra of partial functions. Noting that $x \rest a \compo b \neq 0$ implies $x \rest a \neq 0$ is sufficient to see that the conditions for $\theta(a \compo b)(x/{\sim})$ and $\theta(a) \compo \theta(b)(x/{\sim})$ to be defined coincide.

\smallskip

We now show that each of the additional operations are represented correctly when they are in the signature.

\smallskip

\noindent\underline{$\intersect$} \,To show that $\theta$ represents $\intersect$ correctly, let $a, b \in \algebra{A}$ and $x, y \in \At(\algebra{A})$. Then
\begin{align*}
&(x, y) \in \theta(a \intersect b) \\
\implies \quad & (x, y) \in \theta(a)\text{ and }(x, y) \in \theta(b) & \text{as }a, b \geq a \intersect b \\
\implies \quad & (x, y) \in \theta(a) \intersect \theta(b) 
\shortintertext{and}
& (x, y) \in \theta(a) \intersect \theta(b) \\
\implies \quad & x \compo a = y\text{ and }x \compo b = y \\
\implies \quad & (x \compo a) \intersect (x \compo b) = y \\
\implies \quad & x \compo (a \intersect b) = y &\text{ by \Cref{dist-laws}}\\
\implies \quad & (x, y) \in \theta(a \intersect b)\text{.}
\end{align*}
The proof that $(x/{\sim}, y) \in \theta(a \intersect b) \iff (x/{\sim}, y) \in \theta(a )\intersect\theta( b)$ is similar, using for the right-to-left implication that $(x \rest a) \intersect(x \rest b) = x \rest (a \intersect b)$ is valid for partial functions.
\smallskip

\noindent\underline{$\updatesymb$} \,To show that $\updatesymb$ is represented correctly, let $a, b \in \algebra{A}$ and $x\in \At(\algebra{A})$. 

We first argue that $\theta(\update a b)$ and $ \update{\theta(a)} {\theta(b)}$ agree on $\At(\algebra A)$. Now $\theta(a)$ is defined on $x$ if and only if there is a pair in $x \compo a$. It is easy to see that this happens if and only if there is a pair in $x \compo \update a b$, which is the condition for $\theta(\update a b)$ to be defined on $x$. Hence $\theta(\update a b)$ has the same domain (on $\At(\algebra A)$) as $\theta(a)$ and hence the same domain as $ \update{\theta(a)} {\theta(b)}$. Now suppose $\theta(a)$ is defined on $x$. If $\theta(b)$ is not defined on $x$ then every pair in $x \compo a$ is also in $x \compo \update a b$. This means $x \compo a \leq x \compo \update a b$, then since both are atoms $x \compo a = x \compo \update a b$. So in this case $\theta(a)$ and $\theta(\update a b)$ agree. Hence $ \update{\theta(a)} {\theta(b)}$ and $\theta(\update a b)$ agree. Alternatively, if $\theta(b)$ \emph{is} defined on $x$, then as $x$ is an atom $(x \compo a) \rest x = x = (x\compo b) \rest x$. That is, every point in the image of $x$ is in the domain of both $a$ and $b$. This is sufficient to conclude that $x \compo b \leq x \compo \update a b$, then since both are atoms $x \compo b = x \compo \update a b$. So in this case $\theta(b)$ and $\theta(\update a b)$ agree. Hence $ \update{\theta(a)} {\theta(b)}$ and $\theta(\update a b)$ agree.

We now argue that $\theta(\update a b)$ and $ \update{\theta(a)} {\theta(b)}$ agree on $\At(\algebra A) / {\sim}$ (when appropriate). Similarly to before, it is easy to see that there is a pair in $x \rest a$ if and only if there is a pair in $x \rest \update a b$, and hence $\theta(\update a b)$ and $ \update{\theta(a)} {\theta(b)}$ have the same domains (on $\At(\algebra A) / {\sim}$). Now suppose $\theta(a)$ is defined on $x/{\sim}$. If $\theta(b)$ is not defined on $x/{\sim}$ then every pair in $x \rest a$ is also in $x \rest \update a b$. This means $x \rest a \leq x \rest \update a b$, then since both are atoms $x \rest a = x \rest \update a b$. So in this case $\theta(a)$ and $\theta(\update a b)$ agree. Hence $ \update{\theta(a)} {\theta(b)}$ and $\theta(\update a b)$ agree. Alternatively, if $\theta(b)$ \emph{is} defined on $x/{\sim}$, then as $x$ is an atom $a \rest x = x = b \rest x$. That is, every point in the domain of $x$ is in the domain of both $a$ and $b$. This is sufficient to conclude that $x \rest b \leq x \rest \update a b$, then since both are atoms $x \rest b = x \rest \update a b$. So in this case $\theta(b)$ and $\theta(\update a b)$ agree. Hence $ \update{\theta(a)} {\theta(b)}$ and $\theta(\update a b)$ agree.

\smallskip
\noindent\underline{$\pref$} \,To show that $\pref$ is represented correctly, let $a, b \in \algebra{A}$ and $x\in \At(\algebra{A})$. 

We first argue that $\theta( a \pref b)$ and ${\theta(a)}\pref {\theta(b)}$ agree on $\At(\algebra A)$. We know that there are four possibilities: \begin{itemize}
\item
for every point in the image of $x$, both $a$ and $b$ are defined;
\item
for every point in the image of $x$, the function $a$ is defined but $b$ is not;
\item
for every point in the image of $x$, the function $a$ is undefined, but $b$ is defined;
\item
for every point in the image of $x$, neither $a$ nor $b$ are defined;
\end{itemize}
(and we know there is at least one point in the image of $x$).
In the first two cases, $a \pref b$ is defined for every point in the image of $x$ and agrees with $a$. It follows that $x \compo (a \pref b) = x \compo a$ (and these are nonzero). In the third case $a \pref b$ is defined for every point in the image of $x$ and agrees with $b$. It follows that $x \compo (a \pref b) = x \compo b$ (and these are nonzero, and in this case $x \compo a = 0$). In the fourth case $x \compo a = x \compo b = x \compo (a \pref b) = 0$. In each case we can see that $\theta( a \pref b)$ and ${\theta(a)}\pref {\theta(b)}$ agree on $x$.

The proof that $\theta( a \pref b)$ and $ {\theta(a)} \pref {\theta(b)}$ agree on $\At(\algebra A) / {\sim}$ (when appropriate), is by an identical case analysis.

\smallskip

\noindent\underline{$\D$} \,To show that $\D$ is represented correctly, let $a \in \algebra{A}$ and $x \in \At(\algebra{A})$. Then $0 < \theta(\D(a))(x) = x \compo \D(a) \leq x$ if $\theta(\D(a))(x)$ is defined. Since $x$ is an atom we have, in this case, $\theta(\D(a))(x) = x$. The partial function $\D(\theta(a))$ is also, by definition, a restriction of the identity function. The domains of $\theta(\D(a))$ and $\D(\theta(a))$ are the same, since $\theta(\D(a))(x)$ is defined precisely when $x \compo \D(a) \neq 0$, which is when $x \compo a \neq 0$, which is precisely when $\D(\theta(a))(x)$ is defined. 

\smallskip

\noindent\underline{$\A$} \,To show that $\A$ is represented correctly, let $a \in \algebra{A}$ and $x \in \At(\algebra{A})$. Then $0 < \theta(\A(a))(x) = x \compo \A(a) \leq x$ if $\theta(\A(a))(x)$ is defined. Since $x$ is an atom we have, in this case, $\theta(\A(a))(x) = x$. The partial function $\A(\theta(a))$ is also, by definition, a restriction of the identity function. The domains of $\theta(\A(a))$ and $\A(\theta(a))$ are the same, since we have seen that $\theta(\A(a))(x)$ is defined precisely when $x \compo \A(a) = x$, which is when $x \compo a = 0$, which is precisely when $\A(\theta(a))(x)$ is defined. 

 
 \smallskip
 
 This completes the proof that $\theta$ is a representation. Lastly, we prove the final two claims. 
 
 \smallskip
 
\noindent\eqref{claim1} \,Suppose that composition is completely left-distributive over joins. Let $S$ be a subset of $\algebra{A}$ such that $\join S$ exists. Let $x, y \in \At(\algebra{A})$. Then
\begin{align*}
 & (x, y) \in \bigcup \theta[S] \\
\implies \quad & (x, y) \in \theta(s) &\quad& \text{for some }s \in S \\
\implies \quad & (x, y) \in \theta(\join S) &\quad& \text{as }\join S \geq s
\shortintertext{and similarly for $(x/{\sim},y)$. Conversely,}
 & (x, y) \in \theta(\join S) \\
\implies \quad & x \compo \join S = y & \\
\implies \quad & \join(\{x\} \compo S) = y &\quad& \text{as }\compo\text{ is completely left-distributive over joins} \\
\implies \quad & x \compo s = y &\quad& \text{for some }s \in S\text{, since }y\text{ is an atom} \\
\implies \quad & (x, y) \in \theta(s) &\quad& \text{for some }s \in S \\
\implies \quad &  (x, y) \in \bigcup \theta[S]\text{.}
\end{align*}
The proof that $(x/{\sim}, y) \in \theta(\join S) \implies  (x/{\sim}, y) \in \bigcup \theta[S]$ is similar, using instead left-distributivity of $\rest$ over joins (\Cref{lemma:left_dist_rest}). We conclude that $\theta(\join S) = \bigcup \theta[S]$. 

\smallskip

\noindent\eqref{claim2} \,Suppose instead that composition is completely left-distributive over meets. Let $S$ be a nonempty subset of $\algebra{A}$ such that $\meet S$ exists. Let $x, y \in \At(\algebra{A})$. Then
\begin{align*}
  & (x, y) \in \theta(\meet S) &\quad& \\
\implies \quad & (x, y) \in \theta(s) &\quad& \text{for all }s \in S,\text{ as }\meet S \leq s \\
\implies \quad& (x, y) \in \bigcap \theta[S] 
\shortintertext{and similarly for $(x/{\sim},y)$. Conversely,}
 &  (x, y) \in \bigcap \theta[S]\\
\implies \quad & (x, y) \in \theta(s) &\quad& \text{for all }s \in S \\
\implies \quad & x \compo s = y &\quad& \text{for all }s \in S \\
\implies \quad & \meet(\{x\} \compo S) = y &\\
\implies \quad & x \compo \meet S = y &\quad& \text{as }\compo\text{ is completely left-distributive over meets}  \\
\implies \quad & (x, y) \in \theta(\meet S).
\end{align*}
The proof that $(x/{\sim}, y) \in \bigcap \theta[S] \implies  (x/{\sim}, y) \in \theta(\meet S)$ is similar, using instead left-distributivity of $\rest$ over meets (\Cref{lemma:left_dist_rest_meets}). We conclude that $\theta(\meet S) = \bigcap \theta[S]$.
\end{proof}

We can also use \Cref{prop:rep} to prove that, in contrast to our signatures that contain composition, for our signatures that do \emph{not} contain composition, being atomic is sufficient for a representable algebra to be (meet-)completely representable.

\begin{lemma}\label{lemma:equip}
Let $\{\arest\} \subseteq \sigma \subseteq \{\arest, \intersect, \updatesymb, \pref\}$, and let $\algebra{A}$ be a $\sigma$-algebra. Suppose $\algebra A$ is representable by partial functions and the atoms of $\algebra A$ are separating. Then $\algebra A$ is the $\sigma$-reduct of a $(\sigma \cup \{\compo\})$-algebra that is representable by partial functions, whose atoms are separating, and for which composition is completely left-distributive over meets.
\end{lemma}

\begin{proof}
Define $a \compo b = 0$ for all $a, b \in \algebra A$. It is proved in Section~3 of \cite{10.1093jigpaljzac058} that this yields a $(\sigma \cup \{\compo\})$-algebra that is representable by partial functions. Clearly the atoms of this algebra are separating, because the definition of the ordering is unaffected. It is also clear that composition is completely left-distributive over meets, because both sides of any instance of the law evaluate to $0$.
\end{proof}

\begin{corollary}\label{corollary:representable_and_atomic}
Let $\{\arest\} \subseteq \sigma \subseteq \{\arest, \intersect, \updatesymb, \pref\}$, and let $\algebra{A}$ be a $\sigma$-algebra. Then $\algebra A$ is (meet-)completely representable by partial functions if and only if it is representable by partial functions and its atoms are separating.
\end{corollary}

\begin{proof}
If $\algebra A$ is (meet-)completely representable then it is by definition representable, and by \Cref{cor:separating} its atoms are separating. 

Conversely, if $\algebra A$ is representable and its atoms are separating, then by \Cref{lemma:equip} and \Cref{prop:rep}, it is the $\sigma$-reduct of a $(\sigma \cup \{\compo\})$-algebra that is (meet-)completely representable. Any (meet-)complete representation of an expansion of $\algebra A$ is a (meet\mbox{-)}{\allowbreak}complete representation of $\algebra A$; hence $\algebra A$ is (meet-)completely representable.
\end{proof}

\section{Axiomatising the classes}\label{axiom}

In this final section, we use the conditions for complete representability that we have uncovered to give finite first-order axiomatisations of the complete-representa\-tion classes.

We start with the signatures not containing composition.

\begin{theorem}\label{theorem:no_composition}
Let $\{\arest\} \subseteq \sigma \subseteq \{\arest, \intersect, \updatesymb, \pref\}$. Then the class of $\sigma$-algebras that are (meet-)completely representable by partial functions is a basic elementary class, axiomatisable by a universal-existential-universal first-order sentence.
\end{theorem}

\begin{proof}
By \Cref{corollary:representable_and_atomic}, the (meet-)completely representable algebras are precisely the representable algebras whose atoms are separating. By \Cref{thm:jackson-stokes}, representability can be axiomatised by a finite conjunction of quasiequations. The atoms being separating is (also) a universal-existential-universal first-order property.
\end{proof}

We know from \Cref{axioms} that no existential-universal-existential axiomatisation is possible; hence we have determined the precise amount of quantifier alternation necessary to axiomatise the classes.

To treat the signatures containing composition and intersection, we must first axiomatise the property of being completely left-distributive over \emph{joins}. The following notion is straightforwardly equivalent to the atoms being separating.

\begin{definition}
A poset $\algebra{P}$ is \defn{atomistic} if its atoms are join dense in $\algebra{P}$. That is to say, every element of $\algebra{P}$ is the join of the atoms less than or equal to it.
\end{definition}




\begin{lemma}\label{lemma:phi}
Let $\sigma$ be any signature including $\{\arest,\compo\}$, and let $\algebra{A}$ be a $\sigma$-algebra that is representable by partial functions and whose atoms are separating. Let $\varphi$ be the first-order sentence asserting that for any $a, b, c$, if $c \geq a \compo x$ for all atoms $x$ less than or equal to $b$, then $c \geq a \compo b$. Then composition is completely left-distributive over joins if and only if $\algebra{A} \models \varphi$.
\end{lemma}

\begin{proof}
Suppose first that composition is completely left-distributive over joins. As the atoms of $\algebra{A}$ are separating, $\algebra A$ is atomistic. So for any $a, b \in \algebra{A}$ we have
\[
a \compo b = a \compo \join \{x \in \At(\algebra{A}) \mid x \leq b\} = \join (\{a\} \compo \{x \in \At(\algebra{A}) \mid x \leq b\})
\]
and so $\varphi$ holds.

Now suppose that $\algebra{A} \models \varphi$. Let $a \in \algebra{A}$ and let $S$ be a subset of $\algebra{A}$ such that $\join S$ exists. Then certainly $a \compo \join S$ is an upper bound for $\{a\} \compo S$. To show it is the least upper bound, let $c$ be an arbitrary upper bound for $\{a\} \compo S$. Then
\begin{align*}
&\text{for all }s \in S &\quad& c \geq a \compo s \\
\implies \quad& \text{for all }s \in S\text{ and }x \in \At(\down \join S)\text{ with }x \leq s&& c \geq a \compo x  \\
\implies \quad& \text{for all }x \in \At(\down \join S) && c \geq a \compo x \\
\implies \quad& \text{for all }x \in \At(\algebra{A})\text{ with }x \leq \join S && c \geq a \compo x \\
\implies \quad&&& c \geq a \compo \join S\text{.}
\end{align*}
The third line follows from the second because $x \in \At(\down \join S)$ implies $x \leq s$ for some $s \in S$. To see this, consider the Boolean algebra $\down \join S$. When $x$ is an atom, $x \nleq s$ if and only if $x \intersect s = 0$, which is equivalent to $\compl{x} \geq s$. So if $x \nleq s$ for all $s \in S$ then $\compl{x} \geq \join S$, forcing $x$ to be zero---a contradiction. The fifth line can be seen to follow from the fourth by first writing $\join S$ as the join of the atoms below it and then using $\varphi$.
\end{proof}

We now have everything we need to prove the second of our two main results.

\begin{theorem}\label{theorem:composition}
Let $\{\arest, \compo, \intersect\} \subseteq \sigma \subseteq \{\arest, \compo, \intersect, \updatesymb, \pref, \D, \A\}$. Then the class of $\sigma$-algebras that are completely representable by partial functions is a basic elementary class, axiomatisable by a universal-existential-universal first-order sentence.
\end{theorem}

\begin{proof}
By \Cref{cor:separating}, \Cref{lemma:compdist} and \Cref{prop:rep}, a $\sigma$-algebra is completely representable by partial functions if and only if it is representable by partial functions, its atoms are separating, and composition is completely left-distributive over joins. By \Cref{thm:jackson-stokes}, representability can be axiomatised by a finite conjunction of quasiequations. The atoms being separating is (also) a universal-existential-universal first-order property. By \Cref{lemma:phi}, in the presence of the axioms for the first two properties, the property that composition is completely left-distributive over joins can be written as a first-order sentence, and it is easy to verify that the sentence in question is also universal-existential-universal.
\end{proof}

As before, we know from \Cref{axioms} that no existential-universal-existential axiomatisation is possible, and hence we have determined the precise amount of quantifier alternation necessary to axiomatise the classes.

\section{Open problems}\label{section:problems}

In this section, we mention some open problems relating to complete representability by partial functions. There are too many possible combinations of operations to be worth listing all of them here, so we only highlight some signatures that have a close connection to existing work.

  The first problem concerns the signatures for which \Cref{prop:rep} did not lead to a first-order axiomatisation of the class of meet-completely representable algebras, that is, the eight signatures depicted in \Cref{figure:signatures} that contain composition but not intersection.

\begin{problem}
For signatures $\{\arest, \compo\} \subseteq \sigma \subseteq \{\arest, \compo, \updatesymb, \pref, \D, \A\}$, determine the axiomatisability of the class of $\sigma$-algebras that are meet-completely representable by partial functions.
\end{problem}

  The second problem concerns less expressive signatures than those investigated in this paper. In \cite{10.1093jigpaljzac058}, the base signature is $\{\rest\}$ rather than $\{\arest\}$, and Jackson and Stokes axiomatise the representable algebras for several signatures that cannot express $\arest$. For these signatures, due to results on distributive lattices \cite{egrot}, we cannot expect the notions of meet-complete representation and join-complete representation to coincide or even to be related by an implication.

\begin{problem}\label{problem:less}
For signatures $\{\rest\} \subseteq \sigma \subseteq \{\rest, \compo,\intersect, \updatesymb, \pref, \D\}$, determine the axiomatisability of the class of $\sigma$-algebras that are meet-completely representable by partial functions and the class of $\sigma$-algebras that are join-completely representable by partial functions.
\end{problem}

Note that for a handful of signatures covered by \Cref{problem:less}, even the plain representation class is still to be axiomatised.

The last problem concerns more expressive signatures than those in this paper. The unary \defn{range} operation $\R$ is the operation of taking the diagonal of the range of a function:
\[\R(f) = \{(y, y) \in X^2 \mid y \in \ran(f)\}\text{.}\]
This operation is often studied \cite{1182.20058,hirsch,finiterep}, with several of the representation classes having been axiomatised but none of the complete representation classes. 
Note that if range had been included in our signature then the function $\theta$ in \Cref{prop:rep} would not be a representation, as it would not represent range correctly. \Cref{fig:range} shows how this can happen. The atom $f$ satisfies $f \compo \R(g) = f$ and so $(f, f) \in \theta(\R(g))$, but there is no $h$ such that $h \compo g = f$ and so $(f, f) \not\in \R(\theta(g))$. It is therefore natural to inquire about complete representability for the signatures obtained by adding $\R$ to the signatures resolved in this paper. 

\begin{figure}[H]
\centering
\begin{tikzpicture}
\draw (0,1)--(2.5,0);
\draw[->](0,1)--(1.25,0.5);
\draw[fill] (0,1) circle [radius=0.05];
\node [above] at (1.25,0.55) {$f$};

\draw (0,-1)--(2.5,0);
\draw[->](0,-1)--(1.25,-0.5);
\draw[fill] (0,-1) circle [radius=0.05];
\draw[fill] (2.5,0) circle [radius=0.05];
\node [below] at (1.25,-0.55) {$g$};
\end{tikzpicture}
\caption{Algebra for which $\theta$ does not represent range correctly}\label{fig:range}
\end{figure}
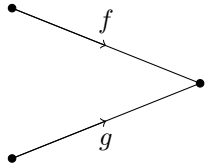

\begin{problem}
For signatures $\{\arest, \R\} \subseteq \sigma \subseteq \{\arest, \R, \intersect, \updatesymb, \pref\}$ or $\{\arest, \compo, \R\} \subseteq \sigma \subseteq \{\arest, \compo, \R, \intersect, \updatesymb, \pref, \D, \A\}$, determine the axiomatisability of the class of $\sigma$-algebras that are (meet-)completely representable by partial functions.
\end{problem}

\bibliographystyle{amsplain}

\bibliography{../brettbib}

\end{document}